\documentclass[10pt,reqno]{amsart}

\usepackage[dvips]{graphicx}

     \makeatletter
     \def\section{\@startsection{section}{1}%
     \z@{.7\linespacing\@plus\linespacing}{.5\linespacing}%
     {\bfseries%\normalfont\scshape
     \centering
     }}
     \def\@secnumfont{\bfseries}
     \makeatother

\usepackage{amsmath}
\usepackage{amssymb}
\usepackage{amsthm}
\usepackage{mathrsfs} 
\usepackage{marvosym}

\usepackage[T1]{fontenc}%

\usepackage{paralist}
\usepackage{color}

\usepackage[verbose]{wrapfig}

\usepackage[ruled,vlined,linesnumbered]{algorithm2e}

\usepackage{pstricks,pst-text,pst-grad,pst-node,pst-3dplot,pstricks-add,pst-poly,pst-coil} 
\usepackage{pst-fun,pst-blur} 

\usepackage{subfigure}
\renewcommand{\thesubfigure}{\thefigure.\arabic{subfigure}}
\makeatletter
\renewcommand{\p@subfigure}{}
\renewcommand{\@thesubfigure}{\thesubfigure:\hskip\subfiglabelskip}
\makeatother

\usepackage{hyperref}
\hypersetup{linktocpage=true,colorlinks=true,linkcolor=blue,citecolor=blue,pdfstartview={XYZ 1000 1000 1}}

\makeatletter
\newcommand{\doublewedge}{\big@doubleop{\wedge}}
\newcommand{\big@doubleop}[1]{%
  \DOTSB\mathop{\mathpalette\big@doubleop@aux{#1}}\slimits@
}

\newcommand\big@doubleop@aux[2]{%
  \sbox\z@{$\m@th#1#2$}%
  \makebox[1.35\wd\z@][s]{$\m@th#1#2\hss#2$}%
}

\newcommand{\norm}[1]{\left\|#1\right\|}  %usage: \norm{x-p} yields ||x-p||

\newcommand{\cl}{\mbox{cl}}
\newcommand{\Int}{\mbox{int}}
\newcommand{\bdy}{\mbox{bdy}}

\newcommand{\Nrv}{\mbox{Nrv}}
\newcommand{\near}{\delta} % Lodato proximity
\newcommand{\dnear}{\delta_{\Phi}} % descriptive proximity
\newcommand{\dcap}{\mathop{\cap}\limits_{\Phi}} % descriptive closure

\newcommand{\sh}{\mbox{sh}}
\newcommand{\cx}{\mbox{cx}}

\newcommand{\dfar}{{\not\delta}_{\Phi}} % descriptive far
\newcommand{\sn}{\mathop{\delta}\limits^{\doublewedge}} % strong proximity

\newcommand{\snd}{\mathop{\delta_{_{\Phi}}}\limits^{\doublewedge}} % snd descriptive proximity
\newtheorem{example}{Example}
\newtheorem{remark}{Remark}
\newtheorem{definition}{Definition}
\newtheorem{lemma}{Lemma}
\newtheorem{theorem}{Theorem}
\newtheorem{proposition}{Proposition}
\newtheorem{corollary}{Corollary}

\usepackage{xcolor}
\definecolor{light}{gray}{0.80}
\begin{document}

\title[Proximal Planar Shapes]{Proximal Planar Shapes.  Correspondence\\ between Shapes and Nerve Complexes}

\author[James F. Peters]{James F. Peters}
\address{James F. PETERS: 
Computational Intelligence Laboratory,
University of Manitoba, WPG, MB, R3T 5V6, Canada and
Department of Mathematics, Faculty of Arts and Sciences, Ad\.{i}yaman University, 02040 Ad\.{i}yaman, Turkey}
\thanks{The research has been supported by the Natural Sciences \&
Engineering Research Council of Canada (NSERC) discovery grant 185986 
and Instituto Nazionale di Alta Matematica (INdAM) Francesco Severi, Gruppo Nazionale per le Strutture Algebriche, Geometriche e Loro Applicazioni grant 9 920160 000362, n.prot U 2016/000036.}
\email{James.Peters3@umanitoba.ca}


\date{}

\dedicatory{Dedicated to P. Alexandroff and Som Naimpally}

\begin{abstract}
This article considers proximal planar shapes in terms of the proximity of shape nerves and shape nerve complexes.   A shape nerve is collection of 2-simplexes with nonempty intersection on a triangulated shape space.   A planar shape is a shape nerve complex, which is a collection of shape nerves that have nonempty intersection.   A main result in this paper is the homotopy equivalence of a planar shape nerve complex and the  union of its nerve sub-complexes.
\end{abstract}

\maketitle

\section{Introduction}
This article introduces shapes that are collections of nerve complexes restricted to the Euclidean plane.

%\setlength{\intextsep}{0pt}
%\begin{wrapfigure}[11]{R}{0.30\textwidth}
%%\fbox{
%\begin{minipage}{3.8 cm}
%%\begin{figure}[!ht]
%\centering
	 %\begin{pspicture}
%[showgrid=true]
%(0.5,-1.5)(4.0,2.0)
%%\psframe[linewidth=2pt,framearc=.3](-0.3,-2.0)(4.8,3.5)
%\rput(2.0,-1.25){\psKangaroo[fillcolor=gray!30,opacity=0.5]{3.5}}
%\pscircle*[linecolor=gray!10,opacity=0.5](3.15,1.20){0.55}
%\pscircle(3.15,1.20){0.55}
%\pscircle*[linecolor=gray!10,opacity=0.5](3.55,1.20){0.55}
%\pscircle(3.55,1.20){0.55}
%\pscircle*[linecolor=gray!10,opacity=0.5](3.3,0.8){0.55}
%\pscircle(3.3,0.8){0.55}
%\psdots[dotstyle=o, linewidth=1.2pt,linecolor = black, fillcolor = black]%
%(3.15,1.20)(3.55,1.20)(3.3,0.8)
%\rput(0.7,1.0){\large $\boldsymbol{\Nrv_1}$}
%\pscircle*[linecolor=gray!50,opacity=0.5](2.50,0.50){0.55}
%\pscircle(2.50,0.50){0.55}
%\pscircle*[linecolor=gray!50,opacity=0.5](2.85,0.50){0.35}
%\pscircle(2.85,0.50){0.55}
%\pscircle*[linecolor=gray!50,opacity=0.5](1.9,0.5){0.55}
%\pscircle(1.9,0.5){0.55}
%\psdots[dotstyle=o, linewidth=1.2pt,linecolor = black, fillcolor = black]%
%(2.50,0.50)(2.85,0.50)(1.9,0.5)
%\rput(2.1,1.8){\large $\boldsymbol{\Nrv_2}$}
%\end{pspicture}
%\label{fig:decomposedShape}}
	%\caption{$\sh \Nrv$}
	%\label{fig:shapeNerve}
%%\end{figure}
%\end{minipage}
%%}
%\end{wrapfigure}  

\setlength{\intextsep}{0pt}
\begin{wrapfigure}[11]{R}{0.30\textwidth}
%\fbox{
\begin{minipage}{3.8 cm}
%\begin{figure}[!ht]
\centering
\begin{pspicture}
%[showgrid=true]
(1.0,-0.2)(3.0,3.8)
%(-0.5,-0.2)(4.0,3.8)
%(-1.0,-0.5)(5.0,3.8)
%(-4.0,-0.5)(5.0,3.8)
%\psset{linewidth=1pt,linecolor=blue}
%\psframe[linecolor=black](-0.0,-0.3)(4.3,4.0)
%\providecommand{\PstPolygonNode}{%
 %\psdots[dotstyle=o,dotsize=0.08,linecolor=blue,fillcolor=yellow](1;\INode)
 %\psline(0.95;\INode)}
%\PstPolygon[unit=1.75,PolyNbSides=8,fillstyle=solid,fillcolor=lightgray]
\rput(1.8,0.25){\psKangaroo[fillcolor=gray!30,opacity=0.5]{3.5}}
\psdot[dotstyle=o,dotsize=0.11,linecolor=black,fillcolor=black](2.3,1.85)
%\rput(-3.8,3.7){$\boldsymbol{X}$}
\rput(3.5,3.2){$\boldsymbol{\sh A}$}
%\rput(-1.0,2.8){$\boldsymbol{\sk B}$}
%\rput(-0.8,1.5){$\boldsymbol{\sk A}$}
\rput(2.52,1.95){$\boldsymbol{p}$}
%\PstPolygon[unit=1.75,PolyNbSides=5,fillstyle=solid,fillcolor=green]
\end{pspicture}
\caption[]{Shape}
%\caption[]{$\mbox{}$\\ {Spokes}}
\label{fig:spokes}
%\end{figure}
\end{minipage}
%}
\end{wrapfigure}  

A \emph{\bf shape} $A$ (denoted by $\sh A$) is a finite region of the Euclidean plane bounded by a simple closed curve with a nonempty interior.  A shape satisfies the Jordan curve theorem.  Recall that a closed curve is \emph{simple}, provided the curve has no self-intersections (loops).

\begin{theorem} {\rm [Jordan Curve Theorem~\cite{Jordan1893coursAnalyse}]}.\\
A simple closed curve lying on the plane divides the\\ 
plane into two regions and forms their common boundary.
\end{theorem}
\begin{proof}
For the first complete proof, see O. Veblen~\cite{Veblen1905TAMStheoryOFPlaneCurves}.  For a simplified proof via the Brouwer Fixed Point Theorem, see R. Maehara~\cite{Maehara1984AMMJordanCurvedTheoremViaBrouwerFixedPointTheorem}.  For an elaborate proof, see J.R. Mundres~\cite[\S 63, 390-391, Theorem 63.4]{Munkres2000}.
\end{proof}

\begin{example}
A shape divides the plane into two regions, providing a common boundary for the two regions\footnote{Many thanks to A. Di Concilio and C. Guadagni for pointing this out.}.  For example, the kangaroo in Fig.~\ref{fig:spokes} is a shape $\sh A$ that includes a point $p$ in its interior.
\qquad \textcolor{blue}{\Squaresteel}
\end{example}

Including both the interior and boundary in a shape reflects the view of shapes by P. Alexandroff and H. Hopf~\cite[\S VII.3, p. 311]{AlexandroffHopf1935Topologie}, who called attention to the importance of both the interior and exterior properties of a shape.

\setlength{\intextsep}{0pt}
\begin{wrapfigure}[9]{R}{0.30\textwidth}
%\fbox{
\begin{minipage}{3.8 cm}
%\begin{figure}[!ht]
\centering
	 \begin{pspicture}
%[showgrid=true]
(0.5,-1.5)(4.0,2.0)
%\psframe[linewidth=2pt,framearc=.3](-0.3,-2.0)(4.8,3.5)
\rput(2.0,-1.25){\psKangaroo[fillcolor=gray!30,opacity=0.5]{3.5}}
\psline[linecolor=blue](0.5,-0.2)(0.8,-0.8)
\psline[linecolor=blue](0.5,-0.2)(1.0,0.5)(2.0,0.80)(2.8,2.1)%
(3.5,2.1)(4.0,1.15)(3.7,0.90)(3.25,0.60)(2.0,0.80)
\psline[linecolor=blue](4.0,0.00)(4.0,1.15)%
(3.7,0.90)(3.25,0.60)(2.0,0.80)(3.75,-1.50)(1.0,0.5)
\psline[linecolor=blue](4.0,0.00)(3.7,0.90)%
\psline[linecolor=blue](4.0,0.00)(3.25,0.60)%
\psline[linecolor=blue](4.0,0.00)(2.0,0.80)%
\psline[linecolor=blue](4.0,0.00)(3.75,-1.50)%
\psline[linecolor=blue](1.0,0.5)(0.8,-0.8)\psline[linecolor=blue](1.0,0.5)(2.0,-1.25)%
\psline[linecolor=blue](2.0,-1.25)(0.8,-0.8)\psline[linecolor=blue](2.0,-1.25)(3.75,-1.50)%
\psline[linecolor=blue](3.25,0.60)(3.5,2.1)\psline[linecolor=blue](3.25,0.60)(2.8,2.1)%
\psline[linecolor=blue](3.25,0.60)(3.5,2.1)\psline[linecolor=blue](3.7,0.90)(3.5,2.1)%
%\psline*[linecolor=green!30,opacity=0.5](3.35,0.60)(3.5,2.1)(2.8,2.1)%
%\psline*[linecolor=green!30,opacity=0.5](3.35,0.60)(3.7,0.90)(3.5,2.1)%
%\psline*[linecolor=green!30,opacity=0.5](3.35,0.60)(3.7,0.90)(4.0,0.00)%
%\psline*[linecolor=green!30,opacity=0.5](3.35,0.60)(2.0,0.80)(4.0,0.00)%
%\psline*[linecolor=green!30,opacity=0.5](3.35,0.60)(2.0,0.80)(2.8,2.1)%
\psdots[dotstyle=o, linewidth=1.2pt,linecolor = black, fillcolor = black]%
(0.5,-0.2)(0.8,-0.8)(1.0,0.5)(2.8,2.1)(3.5,2.1)(4.0,1.15)(3.7,0.90)(3.25,0.60)%
\psdots[dotstyle=o, linewidth=1.2pt,linecolor = black, fillcolor = black]%
(4.0,0.00)(3.75,-1.50)(2.0,-1.25)(2.0,0.80)%
\psdots[dotstyle=o, linewidth=1.3pt,linecolor = black, fillcolor = black]%
(3.25,0.60)
\end{pspicture}
\caption[]{$\boldsymbol{\sh A\subset \left\{\bigtriangleup\right\}}$}
%\caption[]{$\mbox{}$\\ {Spokes}}
\label{fig:kangarooComplex}
%\end{figure}
\end{minipage}
%}
\end{wrapfigure}   

In this paper, a proximal computational topology approach to shapes is introduced.  
Computational topology combines geometry, topology and algorithms in the study of topological structures, introduced by H. Edelsbrunner and J.L. Harer~\cite{Edelsbrunner1999}.  
K. Borsuk was one of the first to suggest studying sequences of plane shapes in his theory of shapes~\cite{Borsuk1970theoryOfShape}.  Borsuk also observed that every polytope can be decomposed into a finite sum of elementary simplexes.  A \emph{polytope} is the intersection of finitely many closed half spaces~\cite{Ziegle2007polytopes}.

\setlength{\intextsep}{0pt}
\begin{wrapfigure}[10]{R}{0.30\textwidth}
%\fbox{
\begin{minipage}{3.8 cm}
%\begin{figure}[!ht]
\centering
	 \begin{pspicture}
%[showgrid=true]
(0.5,-1.8)(4.0,2.5)
%\psframe[linewidth=2pt,framearc=.3](-0.3,-2.0)(4.8,3.5)
\psframe[linewidth=1pt,framearc=.3](0.3,-1.6)(4.3,2.3)
\rput(0.8,2.0){\large $\boldsymbol{K}$}
\rput(2.1,1.3){$\boldsymbol{\Nrv K(p)}$}
\rput(2.0,-1.25){\psKangaroo[fillcolor=gray!30,opacity=0.5]{3.5}}
\psline[linecolor=blue](0.5,-0.2)(0.8,-0.8)
\psline[linecolor=blue](0.5,-0.2)(1.0,0.5)(2.0,0.80)(2.8,2.1)%
(3.5,2.1)(4.0,1.15)(3.7,0.90)(3.25,0.60)(2.0,0.80)
\psline[linecolor=blue](4.0,0.00)(4.0,1.15)%
(3.7,0.90)(3.25,0.60)(2.0,0.80)(3.75,-1.50)(1.0,0.5)
\psline[linecolor=blue](4.0,0.00)(3.7,0.90)%
\psline[linecolor=blue](4.0,0.00)(3.25,0.60)%
\psline[linecolor=blue](4.0,0.00)(2.0,0.80)%
\psline[linecolor=blue](4.0,0.00)(3.75,-1.50)%
\psline[linecolor=blue](1.0,0.5)(0.8,-0.8)\psline[linecolor=blue](1.0,0.5)(2.0,-1.25)%
\psline[linecolor=blue](2.0,-1.25)(0.8,-0.8)\psline[linecolor=blue](2.0,-1.25)(3.75,-1.50)%
\psline[linecolor=blue](3.25,0.60)(3.5,2.1)\psline[linecolor=blue](3.25,0.60)(2.8,2.1)%
\psline[linecolor=blue](3.25,0.60)(3.5,2.1)\psline[linecolor=blue](3.7,0.90)(3.5,2.1)%
\psline*[linecolor=gray!40,opacity=0.5](3.25,0.60)(3.5,2.1)(2.8,2.1)%
\psline*[linecolor=gray!40,opacity=0.5](3.25,0.60)(3.7,0.90)(3.5,2.1)%
\psline*[linecolor=gray!40,opacity=0.5](3.25,0.60)(3.7,0.90)(4.0,0.00)%
\psline*[linecolor=gray!40,opacity=0.5](3.25,0.60)(2.0,0.80)(4.0,0.00)%
\psline*[linecolor=gray!40,opacity=0.5](3.25,0.60)(2.0,0.80)(2.8,2.1)%
\psdots[dotstyle=o, linewidth=1.2pt,linecolor = black, fillcolor = black]%
(0.5,-0.2)(0.8,-0.8)(1.0,0.5)(2.8,2.1)(3.5,2.1)(4.0,1.15)(3.7,0.90)(3.25,0.60)%
\psdots[dotstyle=o, linewidth=1.2pt,linecolor = black, fillcolor = black]%
(4.0,0.00)(3.75,-1.50)(2.0,-1.25)(2.0,0.80)%
\psdots[dotstyle=o, linewidth=1.2pt,linecolor = black, fillcolor = white]%
(3.25,0.60)
\rput(3.5,0.5){$\boldsymbol{p}$}
\end{pspicture}
\caption[]{$\boldsymbol{\Nrv K(p)}$}
%\caption[]{$\mbox{}$\\ {Spokes}}
\label{fig:kangarooNerves}
%\end{figure}
\end{minipage}
%}
\end{wrapfigure}   

This leads to a simplicial complex $K$ (denoted by $\cx K$) covered by 2-simplexes $\Delta_1,\dots,\Delta_n$ (filled triangles) such that the nerve of the decomposition is the same as $K$~\cite{Borsuk1948FMsimplexes}.  Briefly, a \emph{geometric simplicial complex} (denoted by $\Delta(S)$ or simply by $\Delta$) is the convex hull of a set of vertices $S$, {\em i.e.}, the smallest convex set containing $S$.  Geometric simplexes in this paper are restricted to vertices (0-simplexes), line segments (1-simplexes) and filled triangles (2-simplexes) in the Euclidean plane, since our main interest is in the extraction of features of simplexes superimposed on planar digital images.  

%A 2-simplex is a 2-dimensional polytope ({\em i.e.}, a filled\\ 
%triangle) in the Euclidean plane represented by $\mathbb{R}^2$.\\   
A collection $X$ of simplexes that cover a planar\\ 
bounded region is called a \emph{simplicial complex}\\ 
(briefly, \emph{complex}).
The \emph{decomposition} of a finite bounded plane surface region into triangles (surface triangulation) is a covering of the surface region with the vertices, edges and interiors of non-overlapping filled
triangles.  

\begin{example}
A partial triangulation of a kangaroo shape $\sh A$ is given in Fig.~\ref{fig:kangarooComplex}, forming a shape complex in which $\sh A \subset \left\{\bigtriangleup\right\}$ (set of triangles in triangulation).  A more complete triangulation of this shape would include more of the boundary and interior points as triangle vertices.
\qquad \textcolor{blue}{\Squaresteel}
\end{example}

Decomposition triangles are \emph{non-overlapping}, provided the interiors of the triangles are disjoint, {\em i.e.}, the intersection of the triangle interiors is empty~\cite{Meisters1975AMMnonOverlapping}.  Let $X$ be a finite plane surface region, $T$ a set of filled triangles.  Surface $X$ is covered by $T$, provided $X\subseteq T$.  In other words, a \emph{triangulation} of a surface is decomposition of the surface. Each triangulation of a surface region results in a finite system of triangles called a \emph{geometrical complex} (briefly, \emph{complex})~\cite[\S 14]{Alexandroff1932elementaryConcepts}.  A 2-dimensional complex is a finite collection of filled triangles that cover a finite plane surface region.
 
%In this paper, we consider only what is known as a Vietoris-Rips complex, which is a collection of 2-simplices determined by subsets of 3 points in a set of points in the Euclidean plane~\cite{ChambersSilvaEricksonGhrist2010DCGVietorisRipsComplexes}. 

\setlength{\intextsep}{0pt}
\begin{wrapfigure}[11]{R}{0.30\textwidth}
%\fbox{
\begin{minipage}{3.8 cm}
%\begin{figure}[!ht]
\centering
\begin{pspicture}
%[showgrid=true]
(1.0,-0.2)(4.0,3.8)
\providecommand{\PstPolygonNode}{%
 \psdots[dotstyle=o,dotsize=0.08,linecolor=blue,fillcolor=black](1;\INode)
 \psline(0.95;\INode)}
\PstPolygon[unit=1.75,PolyNbSides=3,fillstyle=solid,fillcolor=lightgray]
\psdot[dotstyle=o,dotsize=0.18,linecolor=black,fillcolor=white](-1.75,1.75)
\rput(-1.65,1.53){$\boldsymbol{p}$}
\end{pspicture}
\caption[]{Complex}
%\caption[]{$\mbox{}$\\ {Spokes}}
\label{fig:nucleus}
%\end{figure}
\end{minipage}
%}
\end{wrapfigure}  

An important form of simplicial complex is a collection of simplexes called a nerve (see, {\em e.g.}, Fig.~\ref{fig:kangarooNerves}).  A planar complex $K$ is a \emph{nerve}, provided the simplexes in $K$ have nonempty intersection equal to a vertex (called the nucleus of the nerve).   A nerve of a complex $K$ (denoted by $\Nrv K$) is a collection of 2-simplexes $\bigtriangleup$ in the triangulation of a plane region, defined by\\
\vspace{2mm}

%\[
\qquad $\Nrv K = \left\{\bigtriangleup\in K: \bigcap \bigtriangleup\neq \emptyset\right\}\ \mbox{(Nerve complex)}.$\\
%\]
\vspace{2mm}

Let $K(p)\subset K$ denote a collection of 2-simplexes that includes a vertex $p$.  The \emph{nucleus} $p$ of a nerve complex (denoted by $\Nrv K(p)$) is a vertex $p$ common to the 2-simplexes in $K(p)$.    For simplicity, we also write  $\Nrv K$ (omitting the nucleus $p$).

\begin{example}
The collection of filled triangles in Fig.~\ref{fig:nucleus} is an example of a nerve complex $\Nrv K(p)$ with a nucleus $p$.  That is, $\Nrv K(p)$ is a collection of filled triangles that have vertex $p$ in common.  In effect, for vertex $p$ in shape $\sh A$, a shape nerve
$\Nrv K(p)$ is defined by
\[
\Nrv K(p) = \left\{\bigtriangleup\in K, p\in \sh A: \bigcap \bigtriangleup = p\right\}\ \mbox{(Shape Nerve)}.
\]
A sample shape nerve $\Nrv K(p)$ is shown in Fig.~\ref{fig:kangarooNerves}.
\qquad \textcolor{blue}{\Squaresteel}
\end{example}

Triangulation of point clouds in the plane provides a straightforward basis for the study of shapes covered by nerve complexes.

\begin{example}
Sample overlapping nerve complexes on a kangaroo shape are shown the highlighted regions in Fig.~\ref{fig:kangarooNerves}.  For simplicity, only a partial triangulation of the kangaroo shape is shown.  Some vertices lie outside the shape, reflect the common occurrence of shapes that lie in the interior a bounded plane region. The nerve complexes in Fig.~\ref{fig:kangarooNerves} are maximal, since each one has the highest number of triangles, compared with all other nerve complexes in the triangulation, forming a shape complex.  A more complete triangulation of this shape would include all of the boundary and interior vertices.
\qquad \textcolor{blue}{\Squaresteel}
\end{example}

\setlength{\intextsep}{0pt}
\begin{wrapfigure}[11]{R}{0.30\textwidth}
%\fbox{
\begin{minipage}{3.8 cm}
%\begin{figure}[!ht]
	 \begin{pspicture}
%[showgrid=true]
(0.5,-1.5)(4.0,2.0)
%\psframe[linewidth=2pt,framearc=.3](-0.3,-2.0)(4.8,3.5)
\psframe[linewidth=1pt,framearc=.3](0.3,-1.6)(4.3,2.6)
\rput(0.8,2.3){\large $\boldsymbol{K}$}
\rput(3.3,2.3){\footnotesize $\boldsymbol{\Nrv_1 K(p)}$}
\rput(2.0,-1.25){\psKangaroo[fillcolor=gray!30,opacity=0.5]{3.5}}
\psline[linecolor=blue](0.5,-0.2)(0.8,-0.8)
\psline[linecolor=blue](0.5,-0.2)(1.0,0.5)(2.0,0.80)(2.8,2.1)%
(3.5,2.1)(4.0,1.15)(3.7,0.90)(3.25,0.60)(2.0,0.80)
\psline[linecolor=blue](4.0,0.00)(4.0,1.15)%
(3.7,0.90)(3.25,0.60)(2.0,0.80)(3.75,-1.50)(1.0,0.5)
\psline[linecolor=blue](4.0,0.00)(3.7,0.90)%
\psline[linecolor=blue](4.0,0.00)(3.25,0.60)%
\psline[linecolor=blue](4.0,0.00)(2.0,0.80)%
\psline[linecolor=blue](4.0,0.00)(3.75,-1.50)%
\psline[linecolor=blue](1.0,0.5)(0.8,-0.8)\psline[linecolor=blue](1.0,0.5)(2.0,-1.25)%
\psline[linecolor=blue](2.0,-1.25)(0.8,-0.8)\psline[linecolor=blue](2.0,-1.25)(3.75,-1.50)%
\psline[linecolor=blue](3.25,0.60)(3.5,2.1)\psline[linecolor=blue](3.25,0.60)(2.8,2.1)%
\psline[linecolor=blue](3.25,0.60)(3.5,2.1)\psline[linecolor=blue](3.7,0.90)(3.5,2.1)%
\psline*[linecolor=gray!30,opacity=0.5](3.25,0.60)(3.5,2.1)(2.8,2.1)%
\psline*[linecolor=gray!30,opacity=0.5](3.25,0.60)(3.7,0.90)(3.5,2.1)%
\psline*[linecolor=gray!30,opacity=0.5](3.25,0.60)(3.7,0.90)(4.0,0.00)%
\psline*[linecolor=gray!30,opacity=0.5](3.25,0.60)(2.0,0.80)(4.0,0.00)%
\psline*[linecolor=gray!30,opacity=0.5](3.25,0.60)(2.0,0.80)(2.8,2.1)%
% Nrv2
%\psline*[linecolor=gray!80,opacity=0.5](4.0,0.00)(3.35,0.60)(3.7,0.90)%
\psline*[linecolor=gray!80,opacity=0.5](4.0,0.00)(3.25,0.60)(2.0,0.80)%
\psline*[linecolor=gray!80,opacity=0.5](4.0,0.00)(2.0,0.80)(3.75,-1.50)%
\psline*[linecolor=gray!80,opacity=0.5](2.0,0.80)(3.75,-1.50)(1.0,0.5)%
\psline*[linecolor=gray!80,opacity=0.5](2.0,0.80)(3.25,0.60)(2.8,2.1)%
\rput(1.7,0.95){\textcolor{black}{$\boldsymbol{p'}$}}
\rput(3.1,-0.1){\footnotesize $\boldsymbol{\Nrv_2 K(p')}$}
\psdots[dotstyle=o, linewidth=1.2pt,linecolor = black, fillcolor = black]%
(0.5,-0.2)(0.8,-0.8)(1.0,0.5)(2.8,2.1)(3.5,2.1)(4.0,1.15)(3.7,0.90)(3.25,0.60)%
%(0.5,-0.2)(0.8,-0.8)(1.0,0.5)(2.8,2.1)(3.5,2.1)(4.0,1.15)(3.7,0.90)(3.35,0.60)%
\psdots[dotstyle=o, linewidth=1.2pt,linecolor = black, fillcolor = black]%
(4.0,0.00)(3.75,-1.50)(2.0,-1.25)(2.0,0.80)%
\psdots[dotstyle=o, linewidth=1.5pt,linecolor = black, fillcolor = white]%
(3.25,0.60)
\psdots[dotstyle=o, linewidth=1.5pt,linecolor = black, fillcolor = white]%
(2.0,0.80)
\rput(3.5,0.5){$\boldsymbol{p}$}
\end{pspicture}
%\label{fig:decomposedShape}}
	\caption{\footnotesize $\boldsymbol{\sh A\left(\cx\Nrv\right)}$}
	\label{fig:shapeNerve}
%\end{figure}
\end{minipage}
%}
\end{wrapfigure}  

The study of nerves was introduced by P. Alexandroff~\cite[\S 33, p. 39]{Alexandroff1932elementaryConcepts},\cite{Alexandroff1928dimensionsbegriff}, elaborated by  P. Alexandroff and H. Hopf~\cite[\S 2.1]{AlexandroffHopf1935Topologie}, K. Borsuk~\cite{Borsuk1948FMsimplexes}, J. Leray~\cite{Leray1946homology}, and a number of others such as M. Adamaszek et al.~\cite{Adamaszek2014arXivNerveComplexes}, E.C. de Verdi\`{e}re et al.~\cite{Colin2012multinerves}, H. Edelsbrunner and J.L. Harer~\cite{Edelsbrunner1999}, and more recently by M. Adamaszek, H. Adams, F. Frick, C. Peterson and C. Previte-Johnson~\cite{Adamaszek2016DCGnerveComplexesOfCircularArcs}.  In this paper, a shape nerve-based extension of the Edelsbrunner-Harer Nerve Theorem~\cite[\S III.2, p. 59]{Edelsbrunner1999} is given (see Theorem~\ref{EHnerve}).   A shape nerve complex $\Nrv(\Nrv K(p), p\in \sh A)$ is a collection of shape nerves with nonempty intersection defined by
\begin{align*}
\sh A &= \Nrv\left\{\Nrv K(p):p\in\sh A\subseteq K\ \mbox{and}\ K(p)\subset K\right\}\\
          &=  \left\{\Nrv K(p):\mathop{\bigcap}\limits_{\substack{%
																																											 p\in \sh A,\\
																																											 K(p)\subset K
																																											}}
																																											\Nrv K(p)\neq \emptyset\right\}\ \mbox{(Shape nerve complex)}.
\end{align*}
%Let $K$ be a geometric complex covering a shape $\sh A$ (a finite, bounded region of the plane) and let $B_r(p) \subset K$ a closed ball with radius $r > 0$ and center $p\in \sh A$.  In addition, let $\Nrv B_r(p)$ be a \emph{shape nerve}, which is a collection of closed balls  with nonempty intersection in $\sh A$.   We then define a \emph{shape nerve complex} $\Nrv \left(Nrv B_r(p)\right)$ on geometric shape $\sh A$ to be a collection of shape nerves with nonempty intersection in $\sh A$, {\em i.e.},
%\begin{align*}
%\sh A &\subseteq K\ \mbox{({Shape $\sh A$ in simplicial complex $K$})},\\
%B_r(p)  &= \left\{p'\in K: \norm{p - p'} \leq r\right\}\ \mbox{({closed ball $B_r(p)$, radius $r > 0, p\in \sh A$})},\\
%\Nrv B_r(p) &= \left\{B_r(p): \mathop{\bigcap}\limits_{p\in \sh A} B_r(p)\neq \emptyset\right\}\ \mbox{(Shape nerve)},\\
%\sh A &= \Nrv \left(Nrv B_r(p): p\in \sh A, r > 0\right)\\
        %&=  \left\{\Nrv B_r(p):\mathop{\bigcap}\limits_{\substack{%
																																											 %p\in \sh A
																																											%}}
																																											%\Nrv B_r(p)\subset \sh A\right\}\ \mbox{(Shape nerve complex)}.
%\end{align*}

\begin{example}{\bf Sample Shape Nerve Complex}.\\
A pair of overlapping shape nerves $\Nrv_1 K(p),\Nrv_2 K(q)$ with nuclei $p,p'\in \sh A$ on a shape $\sh A$ are shown in Fig.~\ref{fig:shapeNerve}.    In this case, the shape nerve complex (denoted by $\sh A\left(\cx\Nrv\right)$  in Fig.~\ref{fig:shapeNerve})  equals $\left\{\Nrv_1 K(p),\Nrv_2 K(p')\right\}$, since $\Nrv_1 B_r(p)\cap\Nrv_2 B_{r'}(q)\neq \emptyset$.
\mbox{\qquad \textcolor{blue}{\Squaresteel}}
\end{example}

From a computational topology perspective, homotopy types and homotopic equivalence are introduced in~\cite[\S III.2]{Edelsbrunner1999} and lead to significant results in the theory of nerve complexes.  Let $f,g:X\longrightarrow Y$ be two continuous maps.  A \emph{homotopy} between $f$ and $g$ is a continuous map $H:X\times[0,1]\longrightarrow Y$ so that $H(x,0) = f(x)$ and $H(x,1) = g(x)$.  The sets $X$ and $Y$ are \emph{homotopy equivalent}, provided there are continuous maps $f: X\longrightarrow Y$ and $g:Y\longrightarrow X$ such that $g\circ f \simeq \mbox{id}_X$ and $f\circ g \simeq \mbox{id}_Y$.  This yields an equivalence relation $X\simeq Y$.  In addition, $X$ and $Y$ have the same \emph{homotopy type}, provided $X$ and $Y$ are homotopy equivalent.   A main result in this paper is Theorem~\ref{thm:nerveSpokeTheorem}.
%A result for homotopic equivalence from H. Edelsbrunner and J.L. Harer sheds light on shape complexes.
%
%\begin{theorem}\label{thm:nerveTheorem}{\bf\rm Edelsbrunner-Harer Nerve Theorem~\cite{Edelsbrunner1999}}
%Let $F$ be a finite collection of closed, convex sets in Euclidean space.   Then the nerve of $F$ and the union of the sets in $F$ have the same homotopy type.
%\end{theorem}

%A main result in this paper is the following extension of Theorem~\ref{thm:nerveSpokeTheorem}.

\begin{theorem}\label{thm:nerveSpokeTheorem}
If $\sh A$ is a nerve complex in the Euclidean plane, $\sh A$ is homotopy equivalent to the union of its nerve sub-complexes.
\end{theorem}

%\setlength{\intextsep}{0pt}
%\begin{wrapfigure}[12]{R}{0.30\textwidth}
%%\fbox{
%\begin{minipage}{3.8 cm}
%%\begin{figure}[!ht]
%\centering
%\begin{pspicture}
%%[showgrid=true]
%(1.0,-0.2)(3.0,3.8)
%%(-0.5,-0.2)(4.0,3.8)
%%(-1.0,-0.5)(5.0,3.8)
%%(-4.0,-0.5)(5.0,3.8)
%%\psset{linewidth=1pt,linecolor=blue}
%%\psframe[linecolor=black](-0.0,-0.3)(4.3,4.0)
%%\providecommand{\PstPolygonNode}{%
 %%\psdots[dotstyle=o,dotsize=0.08,linecolor=blue,fillcolor=yellow](1;\INode)
 %%\psline(0.95;\INode)}
%%\PstPolygon[unit=1.75,PolyNbSides=8,fillstyle=solid,fillcolor=lightgray]
%\rput(1.8,0.25){\psKangaroo[linecolor=white,fillcolor=gray!30,opacity=0.5]{3.5}}
%\psdot[dotstyle=o,dotsize=0.11,linecolor=black,fillcolor=black](2.3,1.85)
%%\rput(-3.8,3.7){$\boldsymbol{X}$}
%\rput(3.5,3.2){$\boldsymbol{\sh A}$}
%%\rput(-1.0,2.8){$\boldsymbol{\sk B}$}
%%\rput(-0.8,1.5){$\boldsymbol{\sk A}$}
%\rput(2.52,1.95){$\boldsymbol{p}$}
%%\PstPolygon[unit=1.75,PolyNbSides=5,fillstyle=solid,fillcolor=green]
%\end{pspicture}
%\caption[]{$\boldsymbol{\Int(\sh A)}$}
%%\caption[]{$\mbox{}$\\ {Spokes}}
%\label{fig:shapeInterior}
%%\end{figure}
%\end{minipage}
%%}
%\end{wrapfigure} 

A practical application of shape complexes is the study of the characteristics of surface shapes.  Such shapes can be quite complex when they are found in digital images.  By covering an image object shape with nerve complexes, we simplify the problem of describing object shapes, thanks to a knowledge of geometric features of nerve complexes.  
The problem of classifying shapes is simplified by extracting regional feature values from the nerves complexes that cover each shape.  This is essentially a point-free geometry approach introduced by ~\cite{Peters2016arXivProximalPhysicalGeometry}.

\section{Preliminaries}
This section briefly introduces three basic types of proximities, namely, traditional \emph{spatial proximity} $\near$ and the more recent \emph{strong proximity} $\sn$ and \emph{descriptive proximity} $\snd$ in the study of computational proximity~\cite{Peters2016CP}.

Let $A$ be a nonempty set of vertices, $p\in A$ in a finite, bounded region $X$ of the Euclidean plane.  An \emph{open ball} $B_r(p)$ with radius $r$ is defined by
\[
B_r(p) = \left\{q\in X: \norm{p - q} < r\right\}\ \mbox{(Open ball with center $p$, radius $r$)}.
\]
The \emph{closure} of $A$ (denoted by $\cl A$) is defined by
\[
\cl A = \left\{q\in X: B_r(q)\subset A\ \mbox{for some $r$}\right\}\ \mbox{(Closure of set $A$)}.
\]
The \emph{boundary} of $A$ (denoted by $\bdy A$) is defined by
\[
\bdy A = \left\{q\in X: B(q)\subset A\ \cap\ X\setminus A\right\}\ \mbox{(Boundary of set $A$)}.
\]
Of great interest in the study of shapes is the interior of a shape, found by subtracting the boundary of a shape from its closure.  In general, the \emph{interior} of a nonempty set $A\subset X$ (denoted by $\Int A$) defined by
\[
\Int A = \cl A - \bdy A\ \mbox{(Interior of set $A$)}.
\]

%in Fig.~\ref{fig:kangaroo}

\begin{example}
The interior of a planar shape is the shape without its boundary points.  For example, the gray region minus the contour of the kangaroo shape $\sh A$ shown in Fig.~\ref{fig:spokes}.  In this case, $p\in \Int(\sh A)$.
\qquad \textcolor{blue}{\Squaresteel}
\end{example}

\emph{Proximities} are nearness relations.  In other words, a \emph{proximity} between nonempty sets is a closeness relation between the sets.   A \emph{\bf proximity space} results from endowing a nonempty set with one or more proximities.   Typically, a proximity space is endowed with a proximity from \u Cech~\cite{Cech1966}, Efremovi\u c~\cite{Efremovic1952}, Lodato~\cite{Lodato1962}, Wallman~\cite{Wallman1938}, Naimpally and Warrack~\cite{Naimpally70}, Naimpally~\cite{Naimpally2009} and the more recent descriptive proximity~\cite{Peters2013mcsintro},~\cite{DiConcilio2016arXivDescriptiveProximities}.

\subsection{Spatial Proximity}
A pair of nonempty sets in a proximity space are \emph{spatially near} (\emph{close to each other}), provided the sets have one or more points in common or each set contains one or more points that are sufficiently close to each other.  Let $X$ be a nonempty set, $A,B,C\subset X$. E. \u{C}ech~\cite{Cech1966} introduced axioms for the simplest form of proximity $\delta$, which satisfies

%\noindent \fbox{\bf \u Cech Proximity Axioms~\cite[\S 2.5, p. 439]{Cech1966}}
\begin{description}
\item[{\rm\bf (P1)}] $\emptyset \not\delta A, \forall A \subset X $.
\item[{\rm\bf (P2)}] $A\ \delta\ B \Leftrightarrow B \delta A$.
\item[{\rm\bf (P3)}] $A\ \cap\ B \neq \emptyset \Rightarrow A \near B$.
\item[{\rm\bf (P4)}] $A\ \delta\ (B \cup C) \Leftrightarrow A\ \delta\ B $ or $A\ \delta\ C$. \qquad \textcolor{blue}{$\blacksquare$}
\end{description}

Overloading the symbol $\near$, the Lodato proximity $\delta$~\cite{Lodato1962} satisfies the \u Cech proximity axioms and axiom (P5).
%\\
%\vspace{3mm}

%\noindent \fbox{\bf Lodato Proximity Axiom~\cite{Lodato1962}}
\begin{description}
\item[{\rm\bf (P5)}]  $A\ \delta\ B$ and $\{b\}\ \delta\ C$ for each $b \in B \ \Rightarrow A\ \delta\ C$.
\qquad \textcolor{blue}{$\blacksquare$}
\end{description}

\noindent We can associate a topology with the space $(X, \delta)$ by considering as closed sets those sets that coincide with their own closure.

%Nonempty sets $A,B$ in a topological space $X$ equipped with the proximity $\sn$ are \emph{strongly near} [\emph{strongly contacted}] (denoted $A\ \sn\ B$), provided the sets have at least one point in common.   The strong contact relation $\sn$ was introduced in~\cite{Peters2015JangjeonMSstrongProximity} and axiomatized in~\cite{PetersGuadagni2015stronglyNear},~\cite[\S 6 Appendix]{Guadagni2015thesis} (see, also,~\cite[\S 1.5]{Peters2016CP},~\cite{Peters2015JangjeonMSstrongProximity,PetersGaudagni2015arXivVoronoiManifolds}).

%\noindent \fbox{\bf Strong Proximity~\cite[\S 1.2]{Peters2015AMSJstrongProximity} (see, also,~\cite[\S 1.5]{Peters2016CP},~\cite{Peters2015JangjeonMSstrongProximity,PetersGaudagni2015arXivVoronoiManifolds}).}\\

Let $X$ be a proximity space, $A, B, C \subset X$ and $x \in X$.  The relation $\sn$ on the family of subsets $2^X$ is a \emph{strong proximity}, provided it satisfies the following axioms.

\begin{description}
%\item[{\rm\bf (snN0)}] $\emptyset\ \notfar\ A, \forall A \subset X $, and \ $X\ \sn\ A, \forall A \subset X$.
\item[{\rm\bf (snN1)}] $A \sn B \Leftrightarrow B \sn A$.
\item[{\rm\bf (snN2)}] $A\ \sn\ B$ implies $A\ \cap\ B\neq \emptyset$. 
\item[{\rm\bf (snN3)}] If $\{B_i\}_{i \in I}$ is an arbitrary family of subsets of $X$ and  $A \sn B_{i^*}$ for some $i^* \in I \ $ such that $\Int(B_{i^*})\neq \emptyset$, then $  \ A \sn (\bigcup_{i \in I} B_i)$ 
\item[{\rm\bf (snN4)}]  $\mbox{int}A\ \cap\ \mbox{int} B \neq \emptyset \Rightarrow A\ \sn\ B$.  
\qquad \textcolor{blue}{$\blacksquare$}
\end{description}

\noindent When we write $A\ \sn\ B$, we read $A$ is \emph{strongly near} $B$ ($A$ \emph{strongly contacts} $B$).   For each \emph{strong proximity} (\emph{strong contact}), we assume the following relations:
\begin{description}
\item[{\rm\bf (snN5)}] $x \in \Int (A) \Rightarrow x\ \sn\ A$ 
\item[{\rm\bf (snN6)}] $\{x\}\ \sn\ \{y\}\ \Leftrightarrow x=y$  \qquad \textcolor{blue}{$\blacksquare$} 
\end{description}

%For strong proximity of the nonempty intersection of interiors, we have that $A \sn B \Leftrightarrow \Int A \cap \Int B \neq \emptyset$ or either $A$ or $B$ is equal to $X$, provided $A$ and $B$ are not singletons; if $A = \{x\}$, then $x \in \Int(B)$, and if $B$ too is a singleton, then $x=y$. It turns out that if $A \subset X$ is an open set, then each point that belongs to $A$ is strongly near $A$.  The bottom line is that strongly near sets always share points, which is another way of saying that sets with strong contact have nonempty intersection.   Let $\near$ denote a traditional proximity relation~\cite{Naimpally1970}.

By definition, a planar shape is a nonempty sets of points in the Euclidean plane, since each shape is bounded by a simple closed curve with nonempty interior.    Let $\sh A,\sh B$ be planar shapes, which are nonempty sets in a triangulated finite planar region endowed with the strong proximity $\sn$.   From Axiom (snN2), $\sh A\ \sn\ \sh B$ implies that shape $\sh A$ and shape $\sh B$ overlap.   That is, the finite region of the Euclidean plane represented by $\sh A$ overlaps with the finite planar region represented by $\sh B$, provided the pair of shapes have members in common.

\begin{remark}\label{rem:nerveShape}{Nerve shape}.\\
Let $\Nrv A$ denote a planar nerve, which is a collection of filled triangles $\bigtriangleup$s with a common vertex denoted by $v_{\Nrv A}$ (called the nerve nucleus).     A simple, closed, polygonal curve $\mathcal{C}$ is defined by the sequence of connected $\bigtriangleup$ vertices opposite $v_{\Nrv A}$.    A pair of vertices $v,v'$ are connected, provided there is a sequence of edges that defines a path between $v$ and $v$.   $\mathcal{C}$ is closed, since one can start at any vertex $v$ in $\mathcal{C}$ and traverse $\mathcal{C}$ to reach $v$.   Also, $\mathcal{C}$ has a nonempty interior, since each of the $\bigtriangleup$s in $\Nrv A$ is filled.    Hence, $\mathcal{C}$ defines a polygonal nerve shape (briefly, nerve shape, denoted by $\sh A$).    A sample planar nerve shape is $\Nrv K(p)$ in Fig.~\ref{fig:kangarooNerves}.
\qquad \textcolor{blue}{$\blacksquare$}
\end{remark}

These observations lead to Proposition~\ref{prop:2spoke}.

\begin{proposition}\label{prop:2spoke}
Let $\sh A, \sh B$ be nerve shapes in a triangulated space $K$.\\  $\sh A\ \sn\ \sh B$, if and only if
$\sh A\ \cap\ \sh B = \Delta$ for at least one $\Delta$ common to $\sh A$ and $\sh B$. 
%$\Nrv A\in \sh A, \Nrv B\in \sh B$.
\end{proposition}
\begin{proof}$\mbox{}$\\
$\Rightarrow$: $\sh A\ \sn\ \sh B$ $\Rightarrow$ $\sh A\ \cap\ \sh B\neq \emptyset$ (from Axiom (snN2)).\\
$\Leftarrow$: 
$\sh A\ \cap\ \sh B = \Delta$ for at least one filled $\Delta$ common to $\sh A$ and to $\sh B$.   Hence,
 $\Int(\sh A)\ \cap\ \Int(\sh B) \neq \emptyset \Rightarrow \sh A\ \sn\ \sh B$ (from Axiom (snN4)).
\end{proof}

\begin{corollary}
Let $\sh A, \sh B, \sh E$ be nerve shapes in a triangulated space $K$.
If $(\sh A \cup \sh B) \sn \sh E$, then $\sh A\ \cap\ \sh E = \Delta$ or $\sh B\ \cap\ \sh E = \Delta$
for at least one $\Delta$ common to $\sh A, \sh B$.
%nerve complex $\Nrv K$ common to one of the pairs of shapes.
\end{corollary}
\begin{proof}
Immediate from Prop.~\ref{prop:2spoke} and the definition of a nerve shape in Remark~\ref{rem:nerveShape}.
\end{proof}

\subsection{Descriptive Proximity}
In the run-up to a close look at extracting features from shape complexes, we first consider descriptive proximities introduced in~\cite{Peters2013mcsintro}, fully covered in~\cite{DiConcilio2016arXivDescriptiveProximities} and briefly introduced, here.  There are two basic types of \emph{object features}, namely, \emph{object characteristic} and \emph{object location}.  For example, an object characteristic of a picture point is colour.  
Descriptive proximities resulted from the introduction of the descriptive intersection pairs of nonempty sets~\cite{Peters2013mcsintro},~\cite[\S 4.3, p. 84]{Naimpally2013}. 
%\\
%\vspace{3mm} 

%\noindent \fbox{\bf Descriptive Intersection~\cite{Peters2013mcsintro} and~\cite[\S 4.3, p. 84]{Naimpally2013}.}
\begin{description}
\item[{\rm\bf ($\boldsymbol{\Phi}$)}] $\Phi(A) = \left\{\Phi(x)\in\mathbb{R}^n: x\in A\right\}$, set of feature vectors.
\item[{\rm\bf ($\boldsymbol{\dcap}$)}]  $A\ \dcap\ B = \left\{x\in A\cup B: \Phi(x)\in \Phi(A) \& \in \Phi(x)\in \Phi(B)\right\}$.
\qquad \textcolor{blue}{$\blacksquare$}
\end{description}
%$\mbox{}$\\
%\vspace{3mm}

Let $\Phi(x)$ be a feature vector for $x\in X$, a nonempty set of points.  $A\ \delta_{\Phi}\ B$ reads $A$ is descriptively near $B$, provided $\Phi(x) = \Phi(y)$ for at least one pair of elements, $x\in A, y\in B$.  

\begin{remark}\label{rem:descriptiveIntersection}{Descriptive Intersection of Nerve Complexes}.\\
Let $\Nrv A,\Nrv B$ be a pair of nerve complexes, which are intersecting nerves.   Each nerve is a collection triangles $\bigtriangleup$ with a common vertex.   A description of the nerve complex $\Nrv A$ is defined by
\[
\Phi(\Nrv A) = \left\{\Phi(\bigtriangleup): \bigtriangleup\in \Nrv A\right\}.
\]
For example, the description of $\bigtriangleup\in \Nrv A$ (denoted by $\Phi(\bigtriangleup)$) is defined in terms of a single feature of $\bigtriangleup$, namely, the area of $\bigtriangleup$.   The description of  $\Nrv B$ is defined similarly.    Hence, the description of a a nerve complex is a set of descriptions of its member triangles.    The descriptive intersection of $\Nrv A,\Nrv B$ is defined by
\[
\Nrv A\ \dcap\ \Nrv B = \left\{\bigtriangleup\in \Nrv A\cup \Nrv B: \Phi(\bigtriangleup)\in \Phi(\Nrv A) \& \in \Phi(\bigtriangleup)\in \Phi(\Nrv B)\right\}.
\]
The pair of nerve complexes $\Nrv A,\Nrv B$ have nonempty descriptive intersection, provided there is a $\bigtriangleup\in \Nrv A$ and a $\bigtriangleup'\in \Nrv B$ such that $\Phi(\bigtriangleup) = \Phi(\bigtriangleup')$, {\em i.e.}, the triangles have matching descriptions.
\qquad \textcolor{blue}{$\blacksquare$}
\end{remark}

The proximity $\delta$ in the \u{C}ech, Efremovi\u c, and Wallman proximities is replaced by $\dnear$.  Then swapping out $\near$ with $\dnear$ in each of the Lodato axioms defines a descriptive Lodato proximity~\cite[\S 4.15.2]{Peters2013springer} that satisfies the following axioms.
%\\
%\vspace{3mm}

%\noindent \fbox{\bf Descriptive Lodato Axioms~\cite[\S 4.15.2]{Peters2013springer}}
\begin{description}
\item[{\rm\bf (dP0)}] $\emptyset\ \dfar\ A, \forall A \subset X $.
\item[{\rm\bf (dP1)}] $A\ \dnear\ B \Leftrightarrow B\ \dnear\ A$.
\item[{\rm\bf (dP2)}] $A\ \dcap\ B \neq \emptyset \Rightarrow\ A\ \dnear\ B$.
\item[{\rm\bf (dP3)}] $A\ \dnear\ (B \cup C) \Leftrightarrow A\ \dnear\ B $ or $A\ \dnear\ C$.
\item[{\rm\bf (dP4)}] $A\ \dnear\ B$ and $\{b\}\ \dnear\ C$ for each $b \in B \ \Rightarrow A\ \dnear\ C$. \qquad \textcolor{blue}{$\blacksquare$}
\end{description}

Nonempty sets $A,B$ in a proximity space $X$ are \emph{strongly near} (denoted $A\ \sn\ B$), provided the sets share points.  The notation $A\ \dfar\ B$ reads $A$ is not descriptively near $B$.
 Strong proximity $\sn$ was introduced in~\cite[\S 2]{Peters2015AMSJstrongProximity} and completely axiomatized in~\cite{PetersGuadagni2015stronglyNear} (see, also,~\cite[\S 6 Appendix]{Guadagni2015thesis}).   

\begin{proposition}\label{prop:dnear}
Let $\left(X,\dnear\right)$ be a descriptive proximity space, $A,B\subset X$.  Then $A\ \dnear\ B \Rightarrow A\ \dcap\ B\neq \emptyset$.
\end{proposition}
\begin{proof}
$A\ \dnear\ B \Rightarrow$ there is at least one $x\in A, y\in B$ such that $\Phi(x)=\Phi(y)$ (by definition of $A\ \dnear\ B$).  Hence, $A\ \dcap\ B\neq \emptyset$.
\end{proof}

Next, consider a proximal form of a Sz\'{a}z relator~\cite{Szaz1987}.  A \emph{proximal relator} $\mathscr{R}$ is a set of relations on a nonempty set $X$~\cite{Peters2016relator}.  The pair $\left(X,\mathscr{R}\right)$ is a proximal relator space.  The connection between $\sn$ and $\near$ is summarized in Lemma~\ref{thm:sn-implies-near}.

\begin{lemma}\label{thm:sn-implies-near}
Let $\left(X,\left\{\near,\dnear,\sn\right\}\right)$ be a proximal relator space, $A,B\subset X$.  Then 
\begin{compactenum}[{\rm (}$1${\rm )}]
\item $A\ \sn\ B \Rightarrow A\ \near\ B$.
\item $A\ \sn\ B \Rightarrow A\ \dnear\ B$.
\end{compactenum}
\end{lemma}
\begin{proof}$\mbox{}$\\
{\rm (}$1${\rm )}: From Axiom (snN2), $A\ \sn\ B$ implies $A\ \cap\ B\neq \emptyset$, which implies $A\ \near\ B$ (from Lodato Axiom (P2)).\\
{\rm (}$2${\rm )}: From {\rm (}$1${\rm )}, there are $x\in A, y\in B$\ common to $A$ and $B$.  Hence, $\Phi(x) = \Phi(y)$, which implies $A\ \dcap\ B\neq \emptyset$.  Then, from the descriptive Lodato Axiom (dP2), $A\ \dcap\ B \neq \emptyset \Rightarrow\ A\ \dnear\ B$. This gives the desired result. 
\end{proof}

 Let $2^{2^X}$ denote a collection of sub-collections of a nonempty set $X$.   
Let $\Nrv A$ be a nerve complex.   By definition, $\Nrv A$ is collection of nerves with nonempty intersection.   The boundary of $\Nrv A$ (denoted by $\bdy \Nrv A$) is a sequence of connected vertices.    That is, for each pair of vertices $v,v'\in \bdy \Nrv A$, there is a sequence of edges, starting with vertex $v$ and ending with vertex $v'$.   There are no loops in  $\bdy \Nrv A$.   Consequently, $\bdy \Nrv A$ defines a simple, closed polygonal curve.   The interior of $\bdy \Nrv A$ is nonempty, since $\Nrv A$ is a collection of filled triangles.   Hence, by definition, a $\Nrv A$ is also a nerve shape.   Next, let {\rm\bf ($\boldsymbol{\dcap}$)} be defined in terms of nerve shapes $\Nrv A,\Nrv B$, {\em i.e.},
\[
\Nrv A\ \dcap\ \Nrv B = \left\{\bigtriangleup\in \Nrv A\ \cup\ \Nrv B: \Phi(\bigtriangleup)\in  \Phi(\Nrv A)\ \&\ \Phi(\bigtriangleup)\in  \Phi(\Nrv B)\right\}.
\]

\begin{theorem}\label{thm:spoke}
Let $\left(X,\left\{\dnear,\sn\right\}\right)$ be a proximal relator triangulated space, nerve complexes $\Nrv A,\Nrv B\in 2^{2^X}$.  Then
\begin{compactenum}[{\rm (}$1${\rm )}]
\item $\Nrv A\ \sn\ \Nrv B$ implies $\Nrv A\ \dnear\ \Nrv B$.
\item A triangle $\Delta E\ \in\ \Nrv A\cap \Nrv B$ implies $\Delta E\ \in\ \Nrv A\ \dcap\ \Nrv B$.
\item A triangle $\Delta E\ \in\ \Nrv A\cap \Nrv B$ implies $\Nrv A\ \dnear\ \Nrv B$.
\end{compactenum}
\end{theorem}
\begin{proof}$\mbox{}$\\
{\rm (}$1${\rm )}: Immediate from part {\rm (}$2${\rm )} of Lemma~\ref{thm:sn-implies-near}.\\
{\rm (}$2${\rm )}: By definition, $\Nrv A,\Nrv B$ are nerve shapes.   From Prop.~\ref{prop:2spoke}, $\Delta E\ \in\ \Nrv A\cap \Nrv B$, if and only if $\Nrv A\ \sn\ \Nrv B$.   Consequently, $\Delta E$ is common to $\Nrv A,\Nrv B$.   Then there is a $\Delta\in \Nrv A$ with the same description as a triangle $\Delta\in \Nrv B$.   Let $\Phi(\Delta E)$ be a description of $\Delta E$.   Then, $\Phi(\Delta E)\in \Phi(\Nrv A) \& \in \Phi(\Delta E)\in \Phi(\Nrv B)$, since $\Delta E\ \in\ \Nrv A\cap \Nrv B$.   Hence, $\Delta E\ \in\ \Nrv A\ \dcap\ \Nrv B$ (from Remark~\ref{rem:descriptiveIntersection}).\\
{\rm (}$3${\rm )}: Immediate from {\rm (}$2${\rm )} and Lemma~\ref{thm:sn-implies-near}.
\end{proof}

\begin{corollary}
Let $\left(X,\left\{\dnear,\sn\right\}\right)$ be a proximal relator triangulated space, shapes $\sh A,\sh B\in 2^{2^X}$.  Then
\begin{compactenum}[{\rm (}$1${\rm )}]
\item $\sh A\ \sn\ \sh B$ implies $\sh A\ \dnear\ \sh B$.
\item A nerve complex $\Nrv E\ \in\ \sh A\cap \sh B$ implies $\Nrv E\ \in\ \Nrv A\ \dcap\ \Nrv B$.
\item A nerve complex $\Nrv E\ \in\ \sh A\cap \sh B$ implies $\sh A\ \dnear\ \sh B$.
\end{compactenum}
\end{corollary}

%\begin{example}
%Let $X$ be a topological space endowed with the strong proximity $\sn$ and $A = \left\{(x,0): 0.1\leq x\leq 1\right\}$,$B = \left\{(x,\frac{1}{x}sin(13/x)):0.1\leq x\leq 1\right\}$.  In this case, $A,B$ represented by
%Fig.~\ref{fig:stronglyNearSets} are strongly near sets with many points in common.  
%\qquad \textcolor{blue}{$\blacksquare$}
%\end{example}

The descriptive strong proximity $\snd$ is the descriptive counterpart of $\sn$. 

\begin{definition}\label{def:snd}
Let $X$ be a proximity space, $A, B, C \subset X$ and $x \in X$.  The relation $\snd$ on the family of subsets $2^X$ is a \emph{descriptive strong Lodato proximity}~\cite[\S 4.15.2]{Peters2013springer}, provided it satisfies the following axioms.
\vspace{2mm}

%\noindent \fbox{\bf Descriptive Strong Lodato proximity~\cite[\S 4.15.2]{Peters2013springer}}
\begin{description}
%\item[{\rm\bf (dsnN0)}] $\emptyset\ {\sdfar}\ A, \forall A \subset X $, and \ $X\ \snd\ A, \forall A \subset X$
\item[{\rm\bf (dsnN1)}] $A\ \snd\ B \Leftrightarrow B\ \snd\ A$
\item[{\rm\bf (dsnN2)}] $A\ \snd\ B \Rightarrow\ A\ \dcap\ B \neq \emptyset$
\item[{\rm\bf (dsnN3)}] If $\{B_i\}_{i \in I}$ is an arbitrary family of subsets of $X$ and  $A\ \snd\ B_{i^*}$ for some $i^* \in I \ $ such that $\Int(B_{i^*})\neq \emptyset$, then $A\ \snd\ (\bigcup_{i \in I} B_i)$
\item[{\rm\bf (dsnN4)}] $\Int A\ \dcap\ \Int B \neq \emptyset \Rightarrow A\ \snd\ B$
\qquad \textcolor{blue}{$\blacksquare$}
\end{description}
\end{definition}

\noindent When we write $A\ \snd\ B$, we read $A$ is \emph{descriptively strongly near} $B$.     For each \emph{descriptive strong proximity}, we assume the following relations:

\begin{description}
\item[(dsnN5)] $\Phi(x) \in \Phi(\Int (A)) \Rightarrow x\ \snd\ A$ 
%\item[(dsnN6)] $\{x\}\ \snd\ \{y\} \Leftrightarrow \Phi(x)=\Phi(y)$  \qquad \textcolor{blue}{$\blacksquare$} 
\end{description}

So, for example, if we take the strong proximity related to non-empty intersection of interiors, we have that $A\ \snd\ B \Leftrightarrow \Int A\ \dcap\ \Int B \neq \emptyset$ or either $A$ or $B$ is equal to $X$, provided $A$ and $B$ are not singletons; if $A = \{x\}$, then $\Phi(x) \in \Phi(\Int(B))$, and if $B$ is also a singleton, then $\Phi(x)=\Phi(y)$.
%\\
%\vspace{3mm}

%\begin{example}\label{ex:colourTopology} {\bf Descriptive Strong Proximity}.\\
%Let $X$ be a triangulated space of picture points represented in Fig.~\ref{fig:snTriangles} with red, brown or yellow colors and let $\Phi: X \rightarrow  \mathbb{R}^n $ be a description of $X$ representing the color of a picture point, where $0$ stands for red (r), $1$ for brown (b) and $2$ for yellow (y). Suppose the range is endowed with the topology given by $\tau= \{ \emptyset,\{ r,  b\}, \{r,b,y\} \}$.
%%In Fig.~\ref{fig:EF}, $A,E,C,B\subset X, x\in X$, $\Phi(x) =$ feature vector containing the r,g,b colour intensities of $x$.  
%Then  $\Delta A \ \snd\ \Delta B$, since $\Int \Delta A\ \dcap\ \Int \Delta B\neq \emptyset$, {\em i.e.}, points in the interior of simplexes $\Delta A, \Delta B$ have matching colours.  
%%Similarly, $B \ \notdsn\ C$, since $\Int B\ \dcap\ \Int C \neq \emptyset$.
%\qquad \textcolor{blue}{$\blacksquare$}
%\end{example}

\begin{example}\label{ex:MNCnerve} {Shape Nerves with Descriptive Strong Proximity}.\\
Let $K$ be a planar triangulated region containing shape nerves $\Nrv K(p)$, equipped with the relator $\left\{\sn,\snd\right\}$.  Let $\Phi(\Nrv K(p))$ = wiring of triangles in $\Nrv K(p)$, a single feature description of a shape nerve.  The term \emph{wiring} can be interpreted in different ways.  For example, \emph{shape nerves both with nuclei on a shape boundary} or \emph{overlapping shape nerves with at least one common $\boldsymbol{\bigtriangleup}$}.  For example, the highlighted nerve complexes in~Fig.~\ref{fig:shapeNerve} satisfy both of these wiring conditions.  Let $\Nrv_1 K(p), \Nrv_2 K(p')$ be overlapping nerves in~Fig.~\ref{fig:shapeNerve}.  $\Nrv_1 K(p)\ \sn\ \Nrv_2 K(p')$, since $\Nrv_1 K(p), \Nrv_2 K(p')$ have at least one common $\boldsymbol{\bigtriangleup}$.  Hence, from Lemma~\ref{thm:sn-implies-near}, $\Nrv_1 K(p)\ \dnear\ \Nrv_2 K(p')$.  From Axiom (dsnN4), $\Nrv_1 K(p)\ \snd\ \Nrv_2 K(p')$, since $\Int\left(\Nrv_1 K(p)\right)\ \dcap\ \Int\left(\Nrv_2 K(p')\right)\neq \emptyset$.
\qquad \textcolor{blue}{\Squaresteel}
\end{example}

An easy next step is to consider shape complexes that are descriptively near and descriptively strongly near.  Let $\sh A, \sh B$ be a pair of shape complexes and let $\Nrv K\in \sh A, \Nrv K'\in \sh B$.  Then
$
\sh A\ \dnear\ \sh B,\ \mbox{provided}\ \Nrv K\ \dcap\ \Nrv K'\neq \emptyset,\ \mbox{i.e.},
$
$\Nrv K\ \dnear\ \Nrv K'$.   Taking this a step further, whenever a region in interior of $\sh A$ has a description that matches the description of a region in the interior of $\sh B$, the pair of shapes are descriptively strongly near.  Let $\Nrv K\in \sh A,\Nrv K'\in \sh B$.  Then
\[
\sh A\ \snd\ \sh B,\ \mbox{provided}\ \Int(\Nrv K)\ \dcap\ \Int(\Nrv K')\neq \emptyset.
\]

\begin{theorem}\label{lem:stronglyNearNerves}
Shape complexes with a common nerve are strongly near.
\end{theorem}
\begin{proof}
Immediate from the definition of $\sn$.
\end{proof}

\begin{theorem}\label{lem:descriptivelyNearNerves}
Strongly near shapes are strongly descriptively near.
\end{theorem}
\begin{proof}
Let $\sh A\ \sn\ \sh B$ be strongly near shape complexes.  Then $\sh A, \sh B$ have a nerve complex in common.  Then
$\Int(\sh A)\ \cap\ \Int(\sh B)\neq \emptyset$. Consequently, from Part 2 of Theorem~\ref{thm:spoke}, $\Int(\sh A)\ \dcap\ \Int(\sh B)\neq \emptyset$.  Hence, from Axiom (dsnN4), $\sh A\ \snd\ \sh B$.
\end{proof}

\begin{theorem}\label{lem:stronglyDescriptivelyNearNerves}
Shape complexes containing interior regions with matching descriptions are strongly descriptively near.
\end{theorem}
\begin{proof}
Immediate from the definition of $\snd$.
\end{proof}

\section{Main Results}
Recall that the \emph{nucleus of a nerve complex} is a vertex that is common to the filled triangles in the nerve.

\begin{lemma}\label{lemma:Vertices}
Every vertex of a planar complex with three or more vertices is the nucleus of a nerve.
\end{lemma}
\begin{proof}
Let $\cx K$ be a planar complex with $2k + 1, k\geq 1$ vertices $P$.  We consider only the case where $k = 2$.  Let $p\in P$ and connect $p$ to $q$, one of its neighbouring vertices.  The line seqment $\overline{pq}$ is on the boundary of a half plane.  Repeat this, connecting straight edges from $p$ to each of the remaining vertices in $P$.  Next, connect each $q$ to each of its neighbouring vertices in $\cx K\setminus p$ to form triangles.   Orient the resulting half planes with borders containing the constructed line segments to obtain a collection of filled triangles (denoted by $\Nrv A$).  The vertex $p$ is common to the triangles in $\Nrv A$.  Hence, $p$ is the nucleus of the complex $\Nrv A$.
\end{proof}

\begin{example}
The collection of filled triangles in Fig.~\ref{fig:nucleus} is an example of a nerve complex with a nucleus $p$ (denoted by $\Nrv A(p)$).  That is, $\Nrv A(p)$ is a collection of filled triangles that have vertex $p$ in common.
\qquad \textcolor{blue}{\Squaresteel}
\end{example}

Let $\Nrv K(p), K(p)\subset K, p\in \sh A$ denote a nerve complex with nucleus $p$ in a shape $\sh A$ covered by a complex $K$.

\begin{theorem}
Let $V$ be a nonempty set of nuclei in a triangulated shape $\sh A$ in a complex $\cx K$ with at least one pair of vertices $q,r\in K\setminus \sh A$ for triangles $\bigtriangleup$ that have at least one vertex in $\sh A$.  Then 
\[
\sh A \subseteq \mathop{\bigcup}\limits_{p\in V}\Nrv K(p).
\]
In other words, the nerve complexes of a shape cover the shape.
\end{theorem}
\begin{proof}
If we allow one or more the vertices $q$ in $\bigtriangleup\in \sh A$ to be in $K\setminus \sh A$, the result follows from Lemma~\ref{lemma:Vertices}, since a straight edge $\overline{pq}\in \bigtriangleup\in \sh A$ for at least one straight nerve complex in shape $\sh A$.
\end{proof}

\begin{lemma}\label{lemma:shapeInterior}
Let $\Nrv K(p)$ be a nerved in a shape nerve complex $\sh A\left(\cx\Nrv\right), $.  If $p\in \Int(\sh A)$, then $\Nrv K(p)\ \sn\ \sh A$.
\end{lemma}
\begin{proof}
From Lemma~\ref{lemma:Vertices}, $p$ is the nucleus of the shape nerve $\Nrv K(p)$.  Consequently, $p\in \Int(\Nrv K(p))$.  Then, from Axiom (snN5), $p\ \sn\ \Int(\Nrv K(p))$ and $p\ \sn\ \Int(\sh A)$.  Hence, $\Nrv K(p)$ is strongly near $\sh A$.
\end{proof}

\begin{theorem}
Let $\sh A$ be a shape in complex $\cx K$, $p\in \Int(\sh A)$.  Then
\begin{compactenum}[{\rm (}$1${\rm )}]
\item $\Nrv K(p)\ \sn\ \Int(\sh A)$.
\item $\Nrv K(p)\ \dnear\ \sh A$ if and only if $\Nrv K(p)\ \sn\ \sh A$.
\end{compactenum}
\end{theorem}
%\begin{proof}
%...
%\end{proof}

%Let $\mathscr{F}$ be a finite collection of nerve complexes that cover a space $X$, endowed with the strong proximity $\sn$.  Let $N$ be the nucleus of nerve $\Nrv K$, $K$ a collection of 1-spokes that have in $N$ in common.   Then $\Nrv K$ is defined by
%\[
%\Nrv K = \left\{\sk A\in K: \sk A\ \sn\ N\right\}.
%\]
%
%A nerve complex endowed with a proximal relator is a collection of spokes with proximities given in Lemma~\ref{lem:MNCnerves} and Theorem~\ref{thm:sn-nerve}.

\begin{lemma}\label{lem:MNCnerves}
Let $\sh A\left(\cx\Nrv\right)$ be a shape nerve complex endowed with the strong proximity $\sn$.  Then $\mathop{\bigcap}\limits_{\substack{%
p\in \sh A,\\
K(p)\subset K}} \Nrv K(p)\neq \emptyset$.
\end{lemma}
\begin{proof}
From Lemma~\ref{lemma:Vertices}, every vertex of a triangle $\bigtriangleup$ in a shape nerve $\Nrv K(p)$ on shape $\sh A$ is the nucleus of a nerve.  Consequently, $\Nrv K(p)$ will have a triangle in common with other shape nerves on $\sh A\left(\cx\Nrv\right)$.  Hence, the desired result follows.
\end{proof}

\begin{theorem}\label{thm:sn-nerve}
Let $\left(\sh A\left(\cx\Nrv\right),\left\{\near,\dnear,\sn\right\}\right)$ be a proximal relator space containing shape nerves $\Nrv K(p),\Nrv K(p')\in \sh A\left(\cx\Nrv\right), p,p'\in \Int(\sh A)$.  Then 
\begin{compactenum}[{\rm (}$1${\rm )}]
\item $\Nrv K(p)\ \sn\ \Nrv K(p') \Rightarrow \Nrv K(p)\ \near\ \Nrv K(p')$.
\item $\Nrv K(p)\ \sn\ \Nrv K(p') \Rightarrow \Nrv K(p)\ \dnear\ \Nrv K(p')$.
\end{compactenum}
\end{theorem}
\begin{proof}$\mbox{}$\\
{\rm (}$1${\rm )}: 
%From Lemma~\ref{lem:MNCnerves}, $\mbox{Nrv}\mathscr{F}_{_{MNC}}$ is an Edelsbrunner-Harer nerve.  
$\Nrv K(p)\ \sn\ \Nrv K(p')$ in shape nerve complex $\sh A\left(\cx\Nrv\right)$, {\em i.e.}.\\ $\Nrv K(p),\Nrv K(p')$ overlap.  From Lemma~\ref{lem:MNCnerves}, $\mathop{\bigcap}\limits_{\Nrv K(p)\in \sh A} \Nrv K(p)\neq \emptyset$.\\  Consequently, $\Nrv K(p)\ \sn\ \Nrv K'(p)$.  Then, from Axiom (snN2),\\  $\Nrv K(p)\ \near\ \Nrv K(p')$.\\
\vspace{1mm}

\noindent{\rm (}$2${\rm )}: We consider only $p$ in the description $\Phi(\Nrv K(p))$ of a shape nerve $\Nrv K(p)$.  Shape nerves $\Nrv K(p),\Nrv K(p')$have a vertex $p$ in common, since $\Nrv K(p)\ \sn\ \Nrv K(p')$.  Hence, $\Nrv K(p)\ \dcap\ \Nrv K(p')\neq \emptyset$.  Then, from Lemma~\ref{thm:sn-implies-near}, $\Nrv K(p)\ \dnear\ \Nrv K(p')$. This gives the desired result.
\end{proof}

Let $F$ be a finite collection of sets.   An \emph{Edelsbrunner-Harer nerve}~\cite[\S III.2, p. 59]{Edelsbrunner1999} nerve consists of all nonempty subcollections of $F$ (denoted by $\Nrv F$) whose sets have nonempty intersection, {\em i.e.},
\[
\Nrv F = \left\{X\subseteq F: \bigcap X\neq \emptyset\right\}\ \mbox{(Edelsbrunner-Harer Nerve)}.
\]

\begin{theorem}\label{EHnerve}{\rm ~\cite[\S III.2, p. 59]{Edelsbrunner1999}}{\rm ({Edelsbrunner-Harer Nerve Theorem})}.\\
Let $F$ be a finite collection of closed, convex sets in Euclidean space.  Then the nerve of $F$ and the union of the sets in $F$ have the same homotopy type.
\end{theorem}

%Let $\Nrv B_r(p), p\in \sh A, r > 0$ denote a shape nerve that is a collection of balls $B_r(p)$ with nonempty intersection such that the center $p$ of each ball is a vertex in a shape $\sh A$.    The shape $\sh A$ is a nerve complex defined by a collection of shape nerves $\Nrv B_r(p)$ with nonempty intersection.

\begin{lemma}\label{thm:1spokeHomotopy}
Let shape $\sh A$ be a nerve $\Nrv \left(Nrv K(p): p\in \sh A, r > 0\right)$  defined by the nonempty intersection of a finite collection of shape nerves $\Nrv K(p), p\in \sh A$, which is a finite collection of closed, convex sets in Euclidean space.  Then shape $\sh A$ and the union of the nerves in $\sh A$ have the same homotopy type.
\end{lemma}
\begin{proof} 
From Theorem~\ref{EHnerve}, we have that the union of the shape nerves $\Nrv K(p)\in \sh A$,
 $p\in \sh A$ and $\sh A$ have the same homotopy type.
\end{proof}

\begin{remark}$\mbox{}$\\ 
Every finite, bounded, planar shape $\sh A$ is a nerve complex $\Nrv\left(\Nrv K(p)\right), p\in \sh A$ covered with overlapping shape nerves in $\Nrv\left(\Nrv K(p)\right)$.  The vertex $p$ can either be on the boundary $\bdy(\sh A)$ or in the interior $\Int(\sh A)$ of shape $\sh A$. By considering all such collections that cover a shape $\sh A$, we obtain the main result of this paper, namely, Theorem~\ref{thm:nerveSpokeTheorem} as a straightforward corollary of Lemma~\ref{thm:1spokeHomotopy}.  The results presented in this paper reflect ongoing work on shape theory~\cite{peters2017proximal,AhmadPeters2017arXivDeltaComplexes} and a direct outcome of a seminar on proximal nerve complexes on planar shapes at the University of Salerno during the summer of 2017~\cite{Peters2017UNISAshapeSeminar}. 
\end{remark} 

\begin{remark}{\bf Open Problems}.\\
An open problem in shape theory is covering a planar shape with a curved boundary so that the 2-simplexes in the triangulation conform to the shape curvature.  The rectilinear triangulation approach presented in this paper does not work well for curved shapes such as the kangaroo shape in Fig.~\ref{fig:spokes}.  A step toward the solution of this problem is the introduction of curvilinear triangulation, leading to delta complexes introduced in~\cite{AhmadPeters2017arXivDeltaComplexes}.  

A second open problem in shape theory is the construction of nerve complexes with extensions called spokes that cover a space more effectively than an Alexandroff nerve complex, called a nerve of a system of sets that have nonempty intersection~\cite[\S 33, p. 39]{Alexandroff1932elementaryConcepts}.  A step toward the solution of the nerve spoke problem is the result of recent advances in nerve complexes that are collections of spoke complexes with nonempty intersection, given in~\cite{AhmadPeters2017arXivDeltaComplexes}.

A third open problem in shape theory is the triangulation of space curves (also called twisted curves), intensively studied by D. Hilbert and S. Cohn-Vossen~\cite[\S 27]{Hilbert1932ChelseaSpaceCurves}.  An important related problem is the detection of $\delta$-thin geodesic triangles in nerve complexes in triangulated shapes on either planar or hyperbolic surfaces.  A \emph{geodetic triangle} is $\delta$-thin, provided each of its sides is contained in the $\delta$-neighbourhood of the union of the remaining two sides~\cite[\S 2, p. 70]{Federici2017DCGgeodeticTriangles}.

A fourth open problem in shape theory is the detection of nerve complexes in rectilinear and curvilinear triangulation of object shapes in digital images.  Steps toward the solution to this problem are given in~\cite{peters2017proximal,AhmadPeters2017arXivDeltaComplexes}.

A fifth open problem in shape theory is the detection of nerve complexes and their graph geodesic in rectilinear and curvilinear triangulation of waveforms such as radar wide band signal co-channel interference (see, {\em e.g.},~\cite{Luo2014JASPcochannelInterference}) and brain signals.  Recent work on brain activity is step towards the solution of this problem in neuroscience (see, {\em e.g.},~\cite{Tozzi2017PLRtopodynamics,Tozzi2016CogNeuraldynamicsBrainActivity}).
\qquad \textcolor{blue}{\Squaresteel}
\end{remark}
   
%\begin{theorem}\label{lem:sndMNCnerves}
%Let $2^X$ be a finite collection of strongly near maximal nerve complexes covering a finite region of the Euclidean plane equipped with the relator $\left\{\sn,\snd\right\}$, $\Phi(\Nrv K)$ = number of 1-spokes in $\Nrv K$.   Then $\mathop{\bigcap}\limits_{\Phi}\Nrv K \neq \emptyset$.
%\end{theorem}
%\begin{proof}
%Each maximal $\Nrv K$ has the same description, namely, the number of 1-spokes in the nerve complex.  Hence, $\mathop{\bigcap}\limits_{\Phi}\Nrv K \neq \emptyset$.
%\end{proof}

\bibliographystyle{amsplain}
\bibliography{NSrefs}

\end{document}